\documentclass[twoside]{amsart}
\usepackage{txfonts}
\usepackage{amsrefs}
\usepackage{hyperref}

\newtheorem{theoremmain}{Theorem}
\newtheorem{theorem}{Theorem}[section]
\newtheorem{proposition}[theorem]{Proposition}
\newtheorem{lemma}[theorem]{Lemma}
\newtheorem{corollary}[theorem]{Corollary}
\newcommand{\sgn}{\ensuremath{\mathrm{sgn}}}
\newcommand{\lb}{\ensuremath{\llbracket}}
\newcommand{\rb}{\ensuremath{\rrbracket}}
\newcommand{\noopsort}[1]{}

\DefineSimpleKey{bib}{myurl}
\newcommand\myurl[1]{\url{#1}}
\newcommand\phdize[1]{Ph.D. Thesis, #1}

\BibSpec{manual}{
  +{}{\PrintAuthors} {author}
  +{,}{ \textit} {title}
  +{}{ \parenthesize} {date}
  +{,}{ \myurl} {url}
  +{,}{ } {note}
  +{.}{ } {transition}
}
\BibSpec{misc}{
  +{}{\PrintAuthors} {author}
  +{,}{ \textit} {title}
  +{,}{ } {date}
  +{.}{ \myurl} {url}
  +{,}{ } {note}
  +{.}{ } {transition}
}
\BibSpec{thesis}{
  +{}{\PrintAuthors} {author}
  +{,}{ \textit} {title}
  +{.}{ \phdize } {school}
  +{,}{ } {year}
  +{.}{ } {transition}
}

\title{The 3-adic eigencurve at the boundary of weight space}
\author{David Roe}
\address{David Roe \\ Pacific Institute for the Mathematical Sciences at the University of Calgary}
\email{roed.math@gmail.com}
\date{\today}
\keywords{eigencurve, overconvergent modular forms}
\subjclass[2010]{11F85 (primary), 11F33 (secondary)}

\begin{document}

\begin{abstract}
This paper generalizes work of Buzzard and Kilford \cite{buzzard-kilford:05a} to the case $p=3$, giving an explicit bound for the overconvergence of the quotient $E_\kappa / V(E_\kappa)$ and using this bound to prove that the eigencurve is a union of countably many annuli over the boundary of weight space.
\end{abstract}

\maketitle

\def\CC{\mathbb{C}}
\def\ZZ{\mathbb{Z}}
\def\ZZZ{\mathbb{Z}_3}
\def\FF{\mathbb{F}}
\def\Fp{\mathbb{F}_3}
\def\QQ{\mathbb{Q}}
\def\PP{\mathbb{P}^1}
\def\QQQ{\mathbb{Q}_3}
\def\BT{\mathbf{T}}
\def\EE{\mathbf{E}}
\def\ww{\omega}
\def\Cp{\mathbb{C}_3}
\def\WW{\mathcal{W}}
\def\OO{\mathcal{O}}
\def\OK{\OO_K}
\def\OC{\mathcal{O}_{\mathbb{C}_3}}

\section{Introduction}\label{sec-intro}

This paper grew out of Kevin Buzzard's course A Concrete Introduction to $p$-adic Modular Forms \cite{buzzard:254z}, part of the eigenvarieties semester at Harvard in spring 2006.  It generalizes the results of Buzzard and Kilford \cite{buzzard-kilford:05a} from the case $p = 2$ to $p = 3$.

The eigencurve, first constructed by Coleman and Mazur \cite{coleman-mazur:Eigencurve}, parameterizes eigenvalues of the compact operator $U$ on the space of overconvergent modular forms.  See Emerton's \cite{emerton:thesis} and Smithline's \cite{smithline:thesis} theses for general background on $p$-adic modular forms and the eigencurve.  In this paper, we prove that the $3$-adic eigencurve consists of a countable disjoint union of annuli near the boundary of weight space, and compute the eigenvalues of $U$ on these components of the eigencurve explicitly:

\begin{theoremmain} \label{theorem-main}
If $\kappa$ is a weight corresponding to $w_0 \in \WW$ with $1/3 < |w_0| < 1,$ and if $v = v(w_0)$, then the slopes of $U$ acting on overconvergent modular forms of weight $\kappa$ are the arithmetic progression $0, v, 2v, 3v, 4v, \ldots,$ each appearing with multiplicity 1.
\end{theoremmain}

In Section \ref{sec-prelim}, we introduce notation that we will need, including definitions of the operators $U$ and $V$ and definitions of the modular forms that will play a crucial role in what follows.  In Lemma \ref{fund-lemma}, we prove fundamental relationships between the modular forms just defined.  In the proof of this lemma we used the $q$-expansion principle, GP/PARI \cite{pari} and Sage \cite{sage} in order to obviate the need for a detailed analysis of the poles and zeroes of the various modular forms involved.  The results stated in Lemma \ref{fund-lemma} constitute the part of the paper most likely to fail for other $p$.  Conversely, if such results can be proved for other $p$, most of the rest of the paper would follow.  Section \ref{sec-prelim} concludes with a corollary giving the action of $U$ and $V$ on various power series rings.

Section \ref{sec-fam-T} begins the analysis of families of modular forms.  We analyze $\BT$, a family given by powers of a theta series, in order to gather information about the Eisenstein family.  Using the results of the previous section, we consider various quotients of $\BT, U(\BT), V(\BT)$ and $VU(\BT)$ and prove that these quotients have specific degrees of overconvergence.

We consider the overconvergence of $\EE / V(\EE)$ in Section \ref{sec-fam-EE}, where $\EE$ is the Eisenstein family.  In order to find the degree of overconvergence of $\EE / V(\EE)$, we use a technique suggested by Buzzard that eliminates the need for some of the arguments in \cite{buzzard-kilford:05a}*{\S 4, 5}.  From Coleman and Mazur \cite{coleman-mazur:Eigencurve}, we know that $\EE / V(\EE)$ is at least slightly overconvergent.  We use the fact that $U$ increases overconvergence, together with the explicit overconvergence bounds for the family $\BT$ found in Section \ref{sec-fam-T}, to show that $\EE / V(\EE)$ extends an explicit distance into the supersingular discs.

In Section \ref{sec-EE-reduction} we consider specializations of $\EE / V(\EE)$ to weights $\kappa$ near the boundary of weight space.  If we expand $E_{\kappa} / V(E_{\kappa})$ as a power series in $y$ (a specific parameter on $X_0(9)$ defined in Section \ref{sec-prelim}), then reduce modulo the maximal ideal, the resulting power series over a finite field does not depend on $\kappa$. 

In Section \ref{sec-gen-U} we find a description for the action of $U$ on the $3$-adic Banach space of overconvergent modular forms of weight $\kappa$.  In particular, if $|c|$ is sufficiently close to $1$ then $V(E_{\kappa})(cy)^n$ forms a basis for this Banach space as $n$ ranges over non-negative integers.  We find a generating function that gives us the matrix of $U$ with respect to this basis.

In Section \ref{sec-char-U} we find the valuations of the coefficients of the characteristic power series of $U$.  The coefficients are given by determinants of submatrices of the matrix of $U$.  We use the generating function from Section \ref{sec-gen-U} to find a lower bound on the valuation of the coefficients.  Finally, we use the results of Section \ref{sec-EE-reduction} to prove that this inequality is actually an equality by showing that a certain determinant is a $3$-adic unit.  Knowledge of the valuations of the coefficients then gives us the proof of the main theorem.

Finally, in Section \ref{sec-other} we summarize other work that has been done on the $p=3$ case.

\subsection*{Acknowledgements}
I would like to thank Kevin Buzzard.  His insistence that I work on a project with him in order to get a grade for his class led to this paper, and also allowed me to learn far more from the Eigenvarieties semester at Harvard.  He spent a significant amount of time outside of class helping me understand the material and working with me on the project that eventually became this paper.  

Second, my debt to the paper of Buzzard and Kilford \cite{buzzard-kilford:05a} will be obvious to anyone who has read it.  To a large extent, I follow their structure, their notation and many of their proofs.

\section{Preliminaries}\label{sec-prelim}

Throughout this paper, we will conflate modular forms with their $q$-expansions
in order to make the grammar easier to follow.  Let $\ww$ be a primitive cube
root of unity, and define $K = \QQQ(\ww).$ Set $\OK$ to be the ring of integers of $K$,
let $\pi = \ww - 1$ be a uniformizer for $\OK$, and let $v_3$ be the extension of
the standard valuation on $\QQQ$ to $K$ (ie $v_3(3) = 1$).  Define $u$ such that $3 = u \pi^2$.
Let $\Cp$ be the completion of the algebraic closure of $\QQQ$ and $\OC$ be the
ring of integers of $\Cp$.  On all of these fields, we have the norm $|x| = 3^{-v_3(x)}.$

All rings are commutative with unity, and if $R$ is a ring we define two 
$R$-module homomorphisms $U$ and $V \colon R\lb q\rb  \rightarrow R\lb q\rb $ by:
$$U\left(\sum_{n=0}^\infty r_n q^n \right) = \sum_{n=0}^\infty r_{3n} q^n,$$
and
$$V\left(\sum_{n=0}^\infty r_n q^n \right) = \sum_{n=0}^\infty r_n q^{3n}.$$
One can easily check that $V$ is an $R$-algebra homomorphism.

\begin{lemma}\label{UV-lemma}
For all $g,h \in R\lb q \rb,$ we have $U(gV(h)) = hU(g)$.
\end{lemma}
\begin{proof}
This follows from a straightforward computation.
\end{proof}

\begin{corollary} \label{UV-cor}
If $h \in R\lb q\rb ^\times$, then $V(h)$ is too, and $U(g/V(h)) = U(g)/h.$
\end{corollary}
\begin{proof}
Apply Lemma \ref{UV-lemma} to $g$ and $h^{-1}$ and note that $V$ is a ring homomorphism.
\end{proof}

We now define modular forms that will serve as analogues of those in \cite{buzzard-kilford:05a}*{\S 2} for the $p=3$ case.

For $k \ge 2$ an even integer, define
$$E_k := 1 + \frac{2}{(1-3^{k-1}) \zeta(1-k)} \sum_{n=1}^\infty \Big( \sum_{\substack{0 < d | n \\ 3 \nmid d}} d^{k-1} \Big) q^n,$$
where $\zeta(s)$ is the Riemann zeta function.  Then $E_k$ is a modular form of level
$3$ and weight $k$ obtained, for $k \ge 4$, from the standard level 1 Eisenstein form of
weight $k$ by dropping an Euler factor.  $E_k$ is an eigenform for $U$.

The function
$$\Delta(q) := q \prod_{n=1}^\infty (1-q^n)^{24} = q - 24q^2 + 252q^3 - 1472q^4 + \cdots$$
is a standard level 1 weight 12 modular form.  Set 
$$f = \sqrt{\frac{\Delta(q^3)}{\Delta(q)}} = q + 12q^2 + 90q^3 + 508q^4 + \cdots,$$
a level 3 modular function giving an isomorphism $X_0(3) \rightarrow \mathbb{P}^1$
(this fact follows from the observation that $f = q \prod_{3 \nmid n} (1-q^n)^{-12}$
has a simple zero at the cusp $\infty$ and no other zeroes).

Define
$$\theta := \sum_{(a,b) \in \ZZ^2} q^{a^2+ab+b^2} = 1 + 6q + 6q^3 + 6q^4 + 12q^7 + \cdots,$$
a level 3 weight 1 modular form that will serve many of the same functions that $E_2$ did in \cite{buzzard-kilford:05a}.

\begin{proposition} \label{theta-eigen}
$\theta$ and $\theta^2$ are eigenforms for the $U$ operator.
\end{proposition}
\begin{proof}
Since $\mathcal{M}_2(\Gamma_0(3))$ is 1 dimensional \cite{diamond-shurman:ModularForms}*{Thm. 3.5.1}, and the square of any element of $\mathcal{M}_1(\Gamma_1(3))$ lies in $\mathcal{M}_2(\Gamma_0(3))$, $\mathcal{M}_1(\Gamma_1(3))$ is at most one dimensional.  Thus $\theta$ and $\theta^2$ are both eigenforms.
\end{proof}

Finally, define
$$y = \frac{\frac{\theta}{V(\theta)} - 1}{6} = q - 5q^4 + 32q^7 - 198q^{10} + 1214q^{13} - \cdots,$$
a level 9 modular function giving an isomorphism $X_0(9) \rightarrow \mathbb{P}^1$.

We encapsulate the crucial facts about these modular forms in the following lemma.  Using this lemma, we will then be able to proceed in the same fashion as Buzzard and Kilford in \cite{buzzard-kilford:05a}.

\begin{lemma} \label{fund-lemma} $ $\\
\begin{enumerate}
\item  $U(y) = U(y^2) = 0$ and $U(y^3) = \frac{y (1 + 3y + 9y^2)}{(1+6y)^3}.$
\item For $m \in \ZZ_{\ge 0}$ we have $U(y^{3m+1}) = U(y^{3m+2}) = 0$ and $U(y^{3m}) = \left( \frac{y (1 + 3y + 9y^2)}{(1+6y)^3} \right)^m.$
\item $f = \frac{y (1 + 3y + 9y^2)}{(1-3y)^3},$ and $U(f) = 10 \cdot 3^2 f + 4 \cdot 3^7 f^2 + 3^{11} f^3$ and $V(f) = \frac{y^3}{1-27y^3}$.
\end{enumerate}
\end{lemma}
\begin{proof}
\begin{enumerate}
\item By Proposition \ref{theta-eigen}, $\theta$ and $\theta^2$ are both eigenforms for $U$.  Putting this together with Corollary \ref{UV-cor} and the definition of $y$, we have that
$$6U(y) = U\left(\frac{\theta}{V(\theta)}\right) - 1 = \frac{U(\theta)}{\theta} - 1 = 0$$
and
\begin{align*}
36U(y^2) &= U \left( \left( \frac{\theta}{V(\theta)} - 1 \right)^2 \right)\\
         &= \frac{U(\theta^2)}{\theta^2} - 2\frac{U(\theta)}{\theta} + 1\\
         &= 0.
\end{align*}

In order to show that $U(y^3) = \frac{y (1 + 3y + 9y^2)}{(1+6y)^3},$ one could analyze the zeroes and poles of $U(y^3)$.  But both are meromorphic functions on $X_0(9)$ with at most nine poles, and thus it suffices to check that the first 100 terms of their $q$-expansions agree, which is easily performed on a computer.

\item
The fact that $U(y) = 0$ and $U(y^2) = 0$ implies that $y = q V(F)$ for some $F \in \ZZ[q]$.  Applying Lemma \ref{UV-lemma} we thus have
$U(y^n) = U(q^n V(F)^n) = U(q^n) F^n$, which easily implies $U(y^{3m+1}) = U(y^{3m+2}) = 0$.  On the other hand, $U(y^3) = U(q^3 V(F)^3) = q F^3$, so $U(y^{3m}) = U(q^{3m}V(F)^{3m}) = q^m F^{3m} = U(y^3)^m.$

\item
As in (1), these results follow by a comparison of $q$-expansions.

\end{enumerate}
\end{proof}

Using this lemma, we are able to deduce the following corollary, specifying the image under $U$ and $V$ of various subsets of power series rings.  

\begin{corollary} \label{containment-cor}
Let $R$ be a commutative ring containing $\OK$, let $r \in \OK$ satisfy $v_3(r) \le 1$, and let $R\lb ry\rb $ denote the subring of $R\lb q\rb $ consisting of elements of the form $a_0 + a_1 (ry) + a_2 (ry)^2 + \cdots$.  Then
\begin{enumerate}
\item $R\lb rf\rb  = R\lb ry\rb $ and $rfR\lb rf\rb =ryR\lb ry\rb $.
\item $V(R\lb r^3f\rb ) = R\lb r^3y^3\rb  \subseteq R\lb rf\rb $ and $V(r^3fR\lb r^3f\rb ) = r^3y^3R\lb r^3y^3\rb  \subseteq rfR\lb rf\rb $.
\item $U(R\lb rf\rb ) \subseteq R\lb r^3f\rb $ and $U(rfR\lb rf\rb ) \subseteq r^3fR\lb r^3f\rb $.
\end{enumerate}
\end{corollary}
\begin{proof}
\begin{enumerate}
\item Lemma \ref{fund-lemma}(3) gives $3f = \frac{3y (1 + 3y + 9y^2)}{(1-3y)^3}$ and thus $rf =ry + \cdots \in ryR\lb ry\rb .$  As a power series in $ry$, we can invert this equation and find $ry$ as a power series in $rf$, giving the desired equality.
\item Since $V$ is an $R$-algebra homomorphism, continuous with respect to the $q$-adic topology, and so, again by Lemma \ref{fund-lemma}(3), we have $V(R\lb r^3f\rb ) = R\lb V(r^3f)\rb  = R\lb r^3V(f)\rb  = R\lb r^3y^3\rb .$  In addition, $R\lb r^3y^3\rb  \subseteq R\lb ry\rb  = R\lb rf\rb $.  Finally, if an element of $r^3fR\lb r^3f\rb $ has no constant term, then neither does $V$ of it.
\item By part (1), we have that $R\lb rf\rb  = R\lb ry\rb $.  But $R\lb ry\rb  = R\lb r^3y^3\rb  \oplus ryR\lb r^3y^3\rb  \oplus r^2y^2R\lb r^3y^3\rb $ as $R$-modules, so given $g \in R\lb ry\rb ,$ we can write $g = g_0+g_1+g_2$ with $g_i \in r^iy^iR\lb r^3y^3\rb $.  Then by Lemma \ref{fund-lemma}(1) and Lemma \ref{UV-lemma}, $U(g_1) = U(g_2) = 0$.  Write $g_0 = V(h)$ with $h \in R\lb r^3f\rb $ using part (2).  Then $U(g) = U(g_0) = UV(h) = h$, so $U(R\lb rf\rb ) \subseteq R\lb r^3f\rb $.  If $g \in rfR\lb rf\rb  = ryR\lb ry\rb $ then $g_0 \in r^3y^3R\lb r^3y^3\rb $ and thus we can choose $h \in r^3fR\lb r^3f\rb $, again by part (2).
\end{enumerate}
\end{proof}

\section{Families of Modular Forms} \label{sec-fam-T}
We use weight space to $3$-adically interpolate between modular forms of integral weight.  Define $\WW$ to be the open disc over $\Cp$ with center $0$ and radius $1$.  As in Buzzard and Kilford \cite{buzzard-kilford:05a}, we only want to consider the component of weight space containing the identity.  So we define a \textit{weight} to be a continuous group homomorphism $\kappa \colon \ZZZ^{\times} \rightarrow \Cp^\times$, satisfying $\kappa(-1) = 1$.  The identification of a $\Cp$ point $w \in \WW$ with the unique weight $\kappa$ such that $\kappa(4) = w + 1$ gives a bijection between the set of $\Cp$ points of $\WW$ and the set of all weights.

For $k \in \Cp$ with $|k| < 1$, we can think of $k$ as the weight $x \mapsto x^k$.  In this case, $4^k = w + 1$ and thus $\frac{w}{3} \in k\ZZZ\lb k\rb $.  Therefore, we have
$$\ZZZ\lb w\rb  \subset \ZZZ\lb w/3\rb  \rightarrow \ZZZ\lb k\rb ,$$
where the inclusion is the natural one and the map on the right is the isomorphism sending $w/3$ to $(4^k-1)/3 = k + \cdots \in k\ZZZ\lb k\rb .$

We shall use italics to denote modular forms of fixed weight, and bold face to denote families of modular forms.  We shall consider two families: first $\BT$, defined below, and then $\EE / V(\EE)$, defined in the next section.  We will use $\BT$ to study $\EE / V(\EE)$, our ultimate object of interest.

Define
$$\BT = \theta^k,$$
that is, $\BT$ is the element $\theta^k$ of $1+3kq\ZZZ\lb k,q\rb  \subset \ZZZ\lb k,q\rb ^\times$.  One constructs $\BT$ explicitly using the binomial theorem.  In addition, we have the following application of the binomial theorem that will be used repeatedly in what follows:

\begin{lemma} \label{binomial-lemma}
Let $R$ be a commutative ring containing $\OK$, let $r \in \OK$ be arbitrary, let $\xi$ be an indeterminate, and let $g \in R\lb r\xi\rb $.  Then
$(1 + r\pi \xi g)^k \in 1 + r\pi k \xi R\lb k, r\xi\rb $.
\end{lemma}
\begin{proof}
First note that $v_3(n!) \le (n-1)/2.$  We now use the binomial theorem to conclude that 
\begin{align*}
(1 + r\pi \xi g)^k &= 1 + r\pi k \xi g \left(1 + \frac{k-1}{2!}(r\pi \xi g) + \frac{(k-1)(k-2)}{3!} (r\pi \xi g)^2 + \cdots\right) \\
                   &= 1 + r\pi k \xi g \left(1 + \frac{\pi(k-1)}{2!} (r \xi g) + \frac{\pi^2(k-1)(k-2)}{3!} (r \xi g)^2 + \cdots\right) \\
                   &\in 1 + r\pi k \xi R\lb k, r\xi\rb .
\end{align*}
\end{proof}

We use Lemma \ref{binomial-lemma} to get information about the overconvergence of the family $\BT$.

\begin{lemma} \label{member-lemma} We have the following containments: \\

\begin{enumerate}
\item $\theta / V(\theta) \in 1 + 3f\OK\lb 3f\rb $.
\item $\BT / V(\BT) \in 1 + 3kf\OK\lb k,\pi f\rb $.
\item $U(\BT) / \BT \in 1 + 9kf\OK\lb k,3\pi f\rb $.
\item Let $\sigma$ denote the $\OK\lb k\rb $ algebra automorphism of
$\OK\lb k,q\rb $ sending $q$ to $\omega q$.  Then we have
$\sigma(\BT) / \BT \in 1+3\pi ky\OK\lb k,3y\rb $ and
$\sigma^2(\BT) / \BT \in 1 + 3\pi ky \OK\lb k,3y\rb $.
\item $VU(\BT) / \BT \in 1 + \pi ky \OK\lb k,3y\rb $.  
\item $U(\BT) / VU(\BT) \in 1 + 3 ky \ZZZ\lb k,3y\rb $.
\end{enumerate}
\end{lemma}
\begin{proof}
\begin{enumerate}
\item By the definition of $y$ we have that $\theta / V(\theta) = 1 + 6y$.
But we know from Corollary \ref{containment-cor} (i) that
$3y \in 3f\OK\lb 3f\rb ,$ which immediately implies that
$\theta / V(\theta) \in 1 + 3f\OK\lb 3f\rb .$
\item Write $\theta / V(\theta) = 1 + 3fg$ for $g \in \OK\lb 3f\rb$.
Applying Lemma \ref{binomial-lemma}, we have that
$\BT / V(\BT) \in 1 + 3kf\OK\lb k, \pi f\rb $.
\item Applying Corollary \ref{containment-cor} (iii) with
$R = \OK\lb k\rb $ and $r = \pi$ we have
$U(\BT) / \BT \in 1 + 9kf\OK\lb k, 3\pi f\rb .$
\item Note that $\sigma$ fixes the image of $V$, so
$\frac{\sigma(\BT)}{\BT} = \frac{\sigma(\BT / V(\BT))}{\BT / V(\BT)}.$ 
Now, since the power series of $y$ in terms of $q$ contains only exponents
congruent to 1 modulo 3, $\sigma(y) = \ww y$.  Therefore
$\sigma(\BT / V(\BT)) = \sigma((1 + 6y)^k) = (1+6\ww y)^k$,
and thus $\frac{\sigma(\BT)}{\BT} = \left(\frac{1 + 6\ww y}{1 + 6y}\right)^k = \left(1 + \frac{6\pi y}{1+6y}\right)^k.$  
We now apply Lemma \ref{binomial-lemma}, yielding
$\frac{\sigma(\BT)}{\BT} \in 1 + 3\pi k y \OK\lb k,3y\rb $.

The same argument works with $\sigma$ replaced by $\sigma^2$, noting that $\ww^2-1 = \pi (\ww + 1)$.
\item Since $q^i+\sigma(q^i)+\sigma^2(q^i)$ equals 0 if $i \nequiv 0 \pmod{3}$
and equals 3 if $i \equiv 0 \pmod{3}$, we have that $3VU(\BT) = \BT + \sigma(\BT) + \sigma^2(\BT)$.
Thus $3VU(\BT) / \BT \in 3 + 3\pi k y \OK\lb k, 3y\rb $,
which yields the desired result after division by 3.
\item Part (iii) gives $U(\BT) / \BT \in 1 + 9kf\OK\lb k,3\pi f\rb  \subset 1 + 3ky\OK\lb 1,3y\rb .$ 
Putting this together with part (v) and dividing yields $U(\BT) / VU(\BT) \in 1 + \pi ky\OK\lb k,3y\rb .$
But $U(\BT) / VU(\BT)$ is clearly an element of $\ZZZ\lb k,y\rb $,  and since
$\ZZZ\lb k,y\rb  \cap 1 + \pi ky\OK\lb k,3y\rb  = 1 + 3 ky\ZZZ\lb k,3y\rb $, we have the desired conclusion.
\end{enumerate}
\end{proof}

\section{The Family $\EE / V(\EE)$} \label{sec-fam-EE}

In this section we will prove a result about the degree of overconvergence
of the family of modular functions $\EE / V(\EE)$.
General expositions on families of overconvergent modular functions
and overconvergent modular forms can be found in \citelist{\cite{coleman-mazur:Eigencurve}*{\S 2.1, 2.4} \cite{buzzard-calegari:05a}*{Appendix}}.
For our purposes, however, we may remain at the level of rings,
using only one result from the more general expositions above.
Specifically, we can rephrase Proposition 2.2.7 of Coleman and Mazur for our purposes in the following proposition:

\begin{proposition} \label{EVE-oc-prop}
For all weights $k$, the $p$-adic modular function $E_k / V(E_k) \in \OC\lb rf\rb $ for some $r \in \Cp$ with $|r| < 1$.
\end{proposition}

Using the knowledge that $E_k / V(E_k)$ overconverges, we get the following explicit result on how far $\EE / V(\EE)$ overconverges.  Recall from the beginning of Section \ref{sec-fam-T} that $w = 4^k - 1 \in \WW$.

\begin{theorem} \label{eve-theorem}
$\EE / V(\EE) \in \ZZZ\lb w/3,3y\rb $
\end{theorem}
\begin{proof}
The key idea is to use the fact that $U$ increases overconvergence to prove that something that we know overconverges to a small extent actually overconverges to a much greater degree.  For the moment fix a weight $k$.  Define a map $\tilde{U} \colon \OC\lb f\rb  \rightarrow \OC\lb f\rb $ by 
$$\tilde{U}(\alpha) = U\left(\alpha\frac{U(\theta^k)}{VU(\theta^k)}\right).$$
Note that
\begin{align*} 
\tilde{U}\left(\frac{E_k}{U(\theta^k)}\right) &= U\left(\frac{E_k}{U(\theta^k)} \frac{U(\theta^k)}{VU(\theta^k)}\right) \\
    &= U\left(\frac{E_k}{VU(\theta^k)}\right) \\
    &= \frac{E_k}{U(\theta^k)}.
\end{align*}

Now, if we knew that $E_k / U(\theta^k)$ were an element of $\OC\lb rf\rb $ with $\frac{1}{3} \le |r| < 1$ and $U(\theta^k) / VU(\theta^k) \in \OC\lb 3f\rb $ then we could conclude using Corollary \ref{containment-cor} (iii) that $E_k / U(\theta^k) \in \OC\lb r^3f\rb $ and thus $E_k / U(\theta^k) \in \OC\lb 27f\rb $ by repeated application of $\tilde{U}$. So we need to demonstrate the two assumptions above.  Lemma \ref{member-lemma} (vi) gives $U(\BT) / VU(\BT) \in 1 + 3 ky\ZZZ\lb k,3y\rb $.  Specializing to weight $k$ and using Corollary \ref{containment-cor} (i) yields $U(\theta^k) / VU(\theta^k) \in \OC\lb 3f\rb $.  In addition, by Proposition \ref{EVE-oc-prop} and Lemma \ref{member-lemma} we know that both $E_k / V(E_k)$ and $U(\theta^k) / VU(\theta^k)$ are in $\OC\lb rf\rb $ for some $r$ with $|r| < 1$, and thus so is their quotient, $\frac{E_k / U(\theta^k)}{V(E_k / U(\theta^k))}$.  Therefore, so is $E_k / U(\theta^k)$ and thus we have by the argument above that $E_k / U(\theta^k)$ actually belongs to $\OC\lb 27f\rb $.  Corollary \ref{containment-cor} (ii) now implies that $V(E_k) / VU(\theta^k) \in \OC\lb 3f\rb $.

Putting all of the previous containments together yields
$$\frac{E_k}{V(E_k)} = \frac{E_k}{U(\theta^k)}\frac{U(\theta^k)}{VU(\theta^k)}\frac{VU(\theta^k)}{V(E_k)} \in \OC\lb 3f\rb  = \OC\lb 3y\rb .$$

We now need to work over all weights $k$ simultaneously.  We know that $\EE / V(\EE) \in \ZZZ\lb k, y\rb  = \ZZZ\lb w/3,y\rb $.  Say $\EE / V(\EE) = \sum_{i,j \ge 0} \alpha_{i,j}(w/3)^iy^j$.  Suppose for the sake of contradiction that for some $i$ and $j$, $v_3(\alpha_{i,j}) < j$.  Among such, choose one with minimal $i$ and let $w$ be a weight with $0 < v_3(w/3) < \frac{j - v_3(\alpha_{i,j})}{i}$.  Consider the valuation of the coefficient of $y^j$ in the expansion of $\EE / V(\EE)$:
$$v_3\left(\sum_{m \ge 0} \alpha_{mj} (w/3)^m\right).$$
Note that $v_3(\alpha_{i,j} (w/3)^i) < j$, so the only way that the whole sum could have valuation at least $j$ would be if two terms with low valuation had exactly the same valuation.  But for $m > \frac{j}{v_3(w/3)}$, we have $v_3(\alpha_{mj} (w/3)^m) > j$, so by adjusting $w$ slightly without changing this threshold value of $m$ we can ensure that the minimum valuation occurring in the sum does not appear twice.  This gives a contradiction, since we know that for each weight, $E_k / V(E_k) \in \OC\lb 3y\rb $.
\end{proof}

\begin{corollary} \label{eve-cor}
If we write $\EE / V(\EE) = \sum a_{i,j} w^iy^j$ then $3^{j-i} | a_{i,j}$ for $j \ge i \ge 0$.  
\end{corollary}
\begin{proof}
By the theorem and the fact that $\EE / V(\EE) \in \ZZZ\lb w, q\rb  = \ZZZ\lb w, y\rb $ we can write $\EE / V(\EE) = \sum a_{i,j} w^i y^j = \sum b_{i,j} (w/3)^i(3y)^j$ with $a_{i,j}, b_{i,j} \in \ZZZ$.  Thus $a_{i,j} = 3^{j-i} b_{i,j}$ and the result follows.
\end{proof}

\section{Reduction of the Eisenstein family near the boundary of weight space} \label{sec-EE-reduction}
Let $\FF$ denote the residue field of $\OK$.  As before, write $\EE / V(\EE) = \sum_{i, j} a_{i,j}w^iy^j.$  Now specialize to some weight $w_0 \in \OC$ satisfying $1/3 < |w_0| < 1$, and let $\kappa$ denote the corresponding character.  We deduce that $E_{\kappa} / V(E_{\kappa}) \in \OK\lb w_0 y\rb .$  Write $E_{\kappa} / V(E_{\kappa}) = g_{\kappa}(w_0 y)$ with $g_{\kappa} \in \OK\lb X\rb $.  Let $\bar{g}_{\kappa} \in \FF\lb X\rb $ denote the reduction of $g_{\kappa}$ modulo the maximal ideal of $\OK$.

Define $r(X) \in \FF\lb X\rb $ by $r(X) = \sum_{m \ge 0} X^{3^m}$.  

\begin{lemma} \label{g-bar-lemma}
We have $\bar{g}_{\kappa}(X) = 1 - X^{-1}r(X^3) - X^{-2}(r(X^3) - r(X^3)^2).$  In particular, $\bar{g}_{\kappa}$ is independent of $\kappa$ (for $\kappa$ corresponding to $w_0 \in \WW$ with $1/3 < |w_0| < 1)$.
\end{lemma}
\begin{proof}
Fix $\kappa$ and say $g_{\kappa} = \sum c_n X^n$, with $c_n = c_n(\kappa) \in \OK$.  Specializing $\EE / V(\EE) = \sum_{i,j} a_{i,j}w^iy^j$ to weight $w_0$ we have $c_jw_0^j = \sum_i a_{i,j}w_0^i$ and thus
$$c_j = \sum_i a_{i,j} w_0^{i-j}.$$
Since $|w_0| > 1/3,$ Corollary \ref{eve-cor} implies that $a_{i,j}w_0^{i-j}$ is in the maximal ideal of $\OC$ if $j > i$.  But $a_{i,j} w_0^{i-j}$ is also in the maximal ideal of $\OC$ if $j < i$ since $a_{i,j} \in \ZZZ$ and $|w_0| < 1$.  Therefore,
$$\bar{c}_n = \bar{a}_{n, n} \in \FF.$$
In particular,  $\bar{c}_n$ is independent of the choice of $\kappa$ and thus $\bar{g}_{\kappa}$ is as well.  Thus to finish the proof of the lemma, we need only verify the formula for $\bar{g}_{\kappa}$ for a particular choice of $\kappa$.  Let $\kappa_0$ be the Dirichlet character of conductor 9 given by $\kappa_0(2) = \ww + 1$ where $\ww$ is a primitive cube root of unity.  The weight corresponding to $\kappa_0$ is $\kappa_0(4) - 1 = \ww - 1$ which satisfies $1/3 < |\ww - 1| < 1.$  The corresponding Eisenstein series is 
\begin{align*}
E_{\kappa_0} &= 1 - \Bigg(\frac{1}{18}\sum_{m = 1}^8 m \kappa_0(m)\Bigg)^{-1} \sum_{n > 0} \Bigg(\sum_{\substack{0 < d | n \\ 3 \nmid d}} \kappa_0(d)\Bigg) q^n\\
&= 1 + (1-\ww)q + 3q^2 + (1-\ww)q^3 + (4+2\ww)q^4 + \cdots
\end{align*}
and the corresponding ratio
$$f_0 := E_{\kappa_0}/V(E_{\kappa_0}) = 1 + (1-\ww)q + 3q^2 + (4+5\ww)q^4 + \cdots$$
is a function on $X_0(27)$ which can be checked to satisfy the equation
\begin{multline*}
9y^3f_0^3+(-27y^3-9y^2-3y)f_0^2+((27-27\ww)y^3+27y^2+9y+(2+\ww))f_0 \\
+((-27+27\ww)y^3-27y^2-9y-(2+\ww)) = 0.
\end{multline*}
If we consider $f_0$ as an element of $\OK\lb y\rb $ then this last equation is an identity in $\OK\lb y\rb $.  Dividing the whole equation by $-1-2\ww$ and setting $X = (-1+\ww)y = w_0y$, we deduce that the equation
\begin{multline*}
X^3g_{\kappa_0}(X)^3+(-3X^3+(1-\ww)X^2+\ww X)g_{\kappa_0}(X)^2+((3-3\ww)X^3-(3-3\ww)X^2-3\ww X+\ww)g_{\kappa_0}(X)\\
+((-3+3\ww)X^3+(3-3\ww)X^2+3\ww X-\ww) = 0
\end{multline*}
is an identity in $\OK\lb X\rb $.  Reducing modulo the maximal ideal we find that 
$$X^3\bar{g}_{\kappa_0}(X)^3+X\bar{g}_{\kappa_0}(X)^2+\bar{g}_{\kappa_0}(X)-1 = 0$$
in $\FF\lb X\rb $.  Using the identity $r(X) - r(X)^3 = X$, which holds in $\FF\lb X\rb $, it is straightforward to check that $\bar{g}_{\kappa_0}(X) = 1 - X^{-1}r(X^3) - X^{-2}(r(X^3) - r(X^3)^2)$ is the unique solution to this equation in $\FF\lb X\rb $.  
\end{proof}

\section{Generating Function for the matrix of the $U$-operator near the boundary of weight space.} \label{sec-gen-U}
In this section we begin the computation of the characteristic power series of $U$ acting on overconvergent forms of weight $\kappa$, where $\kappa$ corresponds to a point $w_0$ in weight space with $1/3 < |w_0| < 1$.  In particular, we give an expression for the coefficients of the matrix of $U$ with respect to a certain basis using generating functions.

Almost by definition, $V(E_{\kappa})$ is an overconvergent modular form of weight $\kappa$ \cite{coleman-mazur:Eigencurve}*{Prop. 2.2.7}.  Corollary \ref{containment-cor} (i) implies that if $c \in \Cp$ with $1 > |c| > 1/3$ then the region of $X_0(9)$ defined by $|cy| \le 1$ is isomorphic to the region of $X_0(3)$ defined by $|cf| \le 1$ and thus the powers of $cy$ can be taken as a Banach basis of a 3-adic Banach space $\mathbf{M}_0$ of weight 0 overconvergent modular forms (this space depends on $c$, but we will suppress this choice in our notation).  For $|c|$ sufficiently close to 1, the space $V(E_{\kappa})\mathbf{M}_0$ of overconvergent weight $\kappa$ modular forms will be closed under the action of the standard Hecke operators, and the operator $U$ will be compact.  This space has a Banach basis $\{V(E_{\kappa})(cy)^n:n=0,1,2,\ldots\}$ and we shall prove results about the $U$ operator by analyzing its matrix with respect to this basis.  Define $m_{i,j} \in \Cp$ for $i, j \ge 0$ by
\begin{equation} \label{m-def}
U(V(E_{\kappa})(cy)^j) = V(E_{\kappa})\sum_im_{i,j}(cy)^i.
\end{equation}

\begin{lemma} \label{matrix-lemma}
The generating function $\sum_{i,j \ge 0} m_{i,j}X^iY^j$ is equal to
$$\frac{g_{\kappa}(\frac{w_0}{c}X)(1+\frac{6}{c}X)^3}{(1+\frac{6}{c}X)^3 - Y^3(c^2X+3cX^2+9X^3)}.$$
\end{lemma}
\begin{proof}
A rearrangement of equation \ref{m-def} gives
$$\sum_i m_{i,j}(cy)^i = (E_{\kappa} / V(E_{\kappa}))U((cy)^j).$$
By Lemma \ref{fund-lemma}, $U(y^j) = 0$ if $j$ is not a multiple of $3$, so $m_{i,j} = 0$ in that case.  For $j = 3t,$ we have $U(y^j) = (y(1+3y+9y^2) / (1+6y)^3)^t)$ and thus
$$\sum_i m_{i,j}(cy)^i = g_{\kappa}(w_0y)\left(\frac{c^3y(1+3y+9y^2)}{(1+6y)^3}\right)^t.$$
This is an identity in $\Cp\lb y\rb $, so substituting $X$ for $cy$ gives
$$\sum_i m_{i,j}X^i = g_{\kappa}(w_0 X/c)\left(\frac{c^2X+3cX^2+9X^3}{(1+6X/c)^3}\right)^t.$$
Multiplying by $Y^j$ and summing over $j$ gives 
$$\sum_{i,j} m_{i,j}X^iY^j = g_{\kappa}(w_0 X/c) \sum_{t \ge 0} \left(\frac{(c^2X + 3cX^2+9X^3)Y^3}{(1+6X/c)^3}\right)^t,$$
and summing the geometric series on the right hand side gives the result.
\end{proof}

Since the $m_{i,j}$ are just the matrix coefficients of $U$ operating on the space of weight $\kappa$ overconvergent modular forms, we can read off the well known result that $U$ is compact for $|c| < 1$ sufficiently close to 1 by noting that if $|c| > |w_0|$ then the coefficients of $g_{\kappa}$ are integral and $\frac{w_0}{c}, \frac{6}{c}$ and $c$ all have norm less than 1.

\section{The characteristic power series of $U$ near the boundary of weight space} \label{sec-char-U}

As in the previous section, let $w_0$ satisfy $1/3 < |w_0| < 1$ and let $\kappa$ be the corresponding weight.  In this section we compute the characteristic power series for various compact operators on $p$-adic Banach spaces; see Serre \cite{serre:62a} for the definitions and basic theorems.  Our goal in this section is to determine the valuations of the roots of the characteristic power series of $U$.  In order to do so, we compute the valuations of the coefficients of the characteristic polynomial in Proposition \ref{newton-prop}, then read off the valuations of the roots using Newton polygons.  Since $y$ has the property that $U(y^j) = 0$ when $j$ is not a multiple of three, our matrix is only nonzero on every third row.  In addition, since $U$ is compact we know that the valuations of the rows are increasing.  In Lemma \ref{strip-lemma} we provide the tool to pull off the valuation component of the coefficients of the characteristic polynomial.  Lemmas \ref{matrix3-lemma} and \ref{det-lemma} then give us the ability to prove that what remains has unit determinant.

Fix $s$ a positive integer, and let $d \in \OC$ be nonzero.  Let $N = (n_{i,j})_{0 \le i,j \le 3s-1}$ be a $3s$ by $3s$ matrix with the propoerty that $n_{i,j} \in d^j\OC$ for all $0 \le i, j \le 3s - 1$.  Assume that $n_{i,j} = 0$ when $j$ is not a multiple of $3$.  Let $P(T) = \det(1-TN) = 1 + \cdots = \sum_{\alpha \ge 0} a_{\alpha} T^{\alpha} \in \OC$ denote the ``characteristic power series'' of $N$ (though it is of course actually a polynomial).  For $0 \le \beta \le s$, let $T_{\beta}$ denote the $\beta$ by $\beta$ matrix whose $(i, j)$th entry is $n_{3i,3j}/d^{3j} \in \OC$.

\begin{lemma} \label{strip-lemma}
We have that $a_{\alpha} / d^{3\alpha(\alpha-1)/2} \in \OC$, and furthermore, for $\alpha \le s$ we have that $a_{\alpha} / d^{3\alpha(\alpha-1)/2} \in \OC^{\times}$ iff $\det(T_{\alpha}) \in \OC^{\times}.$
\end{lemma}
\begin{proof}
For $S$ a subset of $\{0,1,2,\ldots, 3s-1\}$ of size $\alpha$, set 
$$d_S = \sum_{\sigma : S \rightarrow S}\sgn(\sigma) \prod_{s \in S} n_{s,\sigma(s)}.$$
By the definition of the determinant, we have that $(-1)^{\alpha}a_{\alpha}$
is the sum of the $d_S$ as $S$ ranges over the size $\alpha$ subsets of $\{0,1,2,\ldots,3s-1\}$.
Note that $d_S = 0$ unless $S$ consists entirely of multiples of $3$.
In this case, $d^{\sum_{s \in S} s}$ divides $d_S$, and 
$\sum_{s \in S} s \ge \frac{3}{2}\alpha(\alpha-1),$ with equality iff 
$S = S_0 := \{0,3,6,\ldots, 3\alpha - 3\}.$  
Thus $a_\alpha$ is a sum of multiples of $d^{3\alpha(\alpha-1)/2}$, 
all but one of which are multiples of $d^{3\alpha(\alpha-1)/2+1}$.  
Therefore $a_{\alpha} / d^{3\alpha(\alpha-1)/2} \in \OC$ and in fact,
$a_{\alpha} / d^{3\alpha(\alpha-1)/2} \in \OC^{\times}$ iff 
$d_{S_0} / d^{3\alpha(\alpha-1)/2} \in \OC^{\times}$.
But $d_{S_0} / d^{3\alpha(\alpha-1)/2} = \det(T_{\alpha})$ and we are done.
\end{proof}

We will use this lemma with $N$ as truncations of the matrix of $U$.  The following lemma allows us to find the coefficients of the matrix $T_\alpha$ in this case.  Recall that $r(X) = \sum_{m \ge 0} X^{3^m}$. 

\begin{lemma} \label{matrix3-lemma}
Define $s_{i,j} \in \Fp$ for $0 \le i,j < \infty$ by
$$\sum_{0 \le i,j} s_{i,j}X^i Y^j = \frac{1 - X^{-1}r(X^3) - X^{-2} (r(X^3) - r(X^3)^2)}{1 - XY^3},$$
with the equality taking place in $\Fp\lb X,Y\rb $.  Define $t_{i,j} = s_{3i,3j}$ for $0 \le i,j$.  Then
$$\sum_{0 \le i,j} t_{i,j}X^iY^j = \frac{1 - r(X)Y + (r(X)^2 - r(X))Y^2}{1 - XY^3}.$$
\end{lemma}
\begin{proof}
Define power series $A(X, Y)$ and $B(X,Y)$ in $\FF\lb X, Y\rb $ by 
$$A(X, Y) = \frac{1 - X^{-1}r(X^3) X^{-2}(r(X^3)^2 - r(X^3))}{1 - XY^3},$$
and
$$B(X, Y) = \frac{1 - r(X)Y + (r(X)^2 - r(X))Y^2}{1 - XY^3}.$$
Our desired result is equivalent to the statement that
$$X^3(A(X, Y) - B(X^3,Y^3)) \cap \FF\lb X^3,Y^3\rb = 0.$$
This follows from explicit computation:
\begin{multline*}
X^3(A(X, Y) - B(X^3,Y^3)) = \\
\frac{(X + X^2 Y^3)r(X^3)^2 - (X + X^2 + X^2 Y^3 + X^4 Y^6) r(X^3) + X^4 Y^3 + X^5 Y^6}{1 - X^3 Y^9}.
\end{multline*}
\end{proof}

Finally, we provide another lemma that allows us to conclude that certain matrices have unit determinant.

\begin{lemma} \label{det-lemma}
Fix an integer $\alpha \ge 0$, and let $\bar{T}_{\alpha}$ be the $\alpha$ by $\alpha$ matrix $(\bar{t}_{i,j})_{0 \le i,j \le \alpha}$ with entries in $\FF$ defined via the following identity:
$$\sum_{i,j \ge 0} \bar{t}_{i,j}X^i Y^j = \frac{1 - r(X) Y + (r(X)^2 - r(X)) Y^2}{1 - XY^3},$$
the equality taking place in $\FF[X,Y]/(X^{\alpha},Y^{\alpha}).$  Then $\det(\bar{T}_{\alpha}) \ne 0$.
\end{lemma}
\begin{proof}
Write $\sum_{i,j} \bar{t}_{i,j}X^i Y^j = \sum_{j} f_j(X) Y^j,$ with $f_j(X) \in V := \FF[X]/(X^{\alpha}).$  It suffices to prove that the $f_j(X), 0 \le j < \alpha,$ span $V$ as an $\FF$-vector space.  Consider $r = r(X)$ as an element of $V$.  We have
$$ \sum_j f_j(X) Y^j = (1 - r Y + (r^2 - r) Y^2)(1 + XY^3 + X^2Y^6 + X^3Y^9 + \cdots)$$
and by comparing powers of $Y$ we see that $f_{3t}(X) = X^t$ and $f_{3t+1} = -r X^t$ and $f_{3t+2} = (r^2 - r) X^t$.  Using the identity $r - r^3 = X$, we have that $f_{3t} = (r - r^3)^t$ and $f_{3t+1} = -r (r - r^3)^t$ and $f_{3t+2} = (r^2 - r)(r - r^3)^t$ and hence as polynomials in $r$ we have that $\deg(f_n) = n$.  Therefore the span of the $f_j$ contains the image of $\FF[r]$ in $V$.  This is enough because $r = X + \cdots$, so this image is all of $\FF[X]/(X^{\alpha}).$
\end{proof}

We now prove a proposition that gives the valuations of the coefficients of the characteristic power series of $U$.  As usual let $\kappa$ be a weight such that the corresponding $w_0$ satisfies $1/3 < |w_0| < 1$, and let $(m_{i,j})$ be the matrix representing $U$ in weight $\kappa$.

\begin{proposition} \label{newton-prop}
If $P_{\kappa}(T) = \sum_{\alpha \ge 0} b_{\alpha} T^{\alpha}$ denotes the characteristic power series of $U$ in weight $\kappa$, then $|b_{\alpha}| = |w_0|^{\alpha(\alpha-1) / 2}$.
\end{proposition}
\begin{proof}
If $\beta \ge 0$ and $M_{\beta}$ denotes the truncated matrix $(m_{i,j})_{0 \le i,j < \beta}$, and if $P_{\beta}(T) = \det(1 - TM_{\beta})$ is the characteristic power series of $M_{\beta}$, then the $P_{\beta}(T)$ tend to $P_{\kappa}(T)$ in the sense that if $P_{\beta}(T) = \sum_{\alpha} b_{\alpha, \beta} T^{\alpha}$ then $\lim_{\beta \rightarrow \infty} b_{\alpha, \beta} = b_{\alpha}$.  Therefore it suffices to prove that $|b_{\alpha, \beta}| = |w_0|^{\alpha(\alpha-1)/2}$ for $\beta > 3\alpha$, and we may further assume that $\beta$ is a multiple of $3$.  Let $N_{\beta}$ be the matrix with elements $(n_{i,j})_{0 \le i,j < \beta}$ where $n_{i,j} = m_{i,j}(c / w_0)^{i-j}$.  Then $N_{\beta}$ is easily checked to be a conjugate of $M_{\beta}$, so $P_{\beta}(T) = det(1 - TN_{\beta})$.  Furthermore, one easily checks that Lemma \ref{matrix-lemma} implies (substituting $X$ for $w_0/c X$ and $Y$ for $c/w_0 Y$)
$$ F(X,Y) := \sum_{0 \le i,j < \beta} n_{i,j}X^i Y^j = \frac{g_{\kappa}(X)(1 + 6/w_0 X)^3}{(1 + 6/w_0 X)^3 - Y^3(w_0^2 X + 3w_0 X^2 + 9X^3)},$$
as an element of $\OC[X,Y]/(X^{\beta},Y^{\beta})$.  Choose $d \in \OC$ with $d^3 = w_0^2.$  The fact that $G(X, Y) := F(X, Y/d)$ satisfies
$$G(X,Y) = \frac{g_{\kappa}(X)(1 + 6/w_0 X)^3}{(1 + 6/w_0 X)^3 - Y^3(X + 3/w_0 X^2 + 9/w_0^2 X^3)}$$
shows that $n_{i,j} / d^j \in \OC$ for all $i,j,$ and the fact that $F(X,Y)$ is a function of $X$ and $Y^3$ implies that $n_{i,j} = 0$ if $j$ is not a multiple of $3$.  We are therefore in position to apply Lemma \ref{strip-lemma} to deduce that $|b_{\alpha,\beta}| \le d^{3\alpha(\alpha-1)/2} = |w_0|^{\alpha(\alpha-1)/2},$ with equality iff the matrix $(n_{3i,3j}/d^{3j})_{0 \le i,j < \beta}$ has unit determinant.  Let $T_{\alpha}$ denote this matrix, and let $\bar{T}_{\alpha}$ denote its reduction modulo the maximal ideal of $\OC$.  Reducing $G(X,Y)$ modulo the maximal ideal of $\OC$, it becomes
$$\bar{G}(X, Y) = \frac{\bar{g}_{\kappa}(X)}{1 - XY^3} \in \FF[X,Y] / (X^{\beta}Y^{\beta})$$
and by Lemma \ref{g-bar-lemma} and Lemma \ref{matrix3-lemma} we deduce that $\bar{T}_{\alpha} = (\bar{t}_{i,j})_{0 \le i,j < \alpha}$ with 
$$\sum_{0 \le i,j < \alpha} \bar{t}_{i,j} X^i Y^j = \frac{1 - r(X) Y + (r(X)^2 - r(X)) Y^2}{1 - XY^3},$$
the equality taking place in $\FF[X,Y]/(X^{\alpha}Y^{\alpha}).$  Now Lemma \ref{det-lemma} implies that $\det(\bar{T}_{\alpha})$ is nonzero, and hence that $\det(T_{\alpha}) \in \OC^{\times}$.  The second part of Lemma \ref{strip-lemma} now implies the desired equality.
\end{proof}

This proposition allows us to prove Theorem \ref{theorem-main}.

\setcounter{theoremmain}{0}
\begin{theoremmain}
If $\kappa$ is a weight corresponding to $w_0 \in \WW$ with $1/3 < |w_0| < 1,$ and if $v = v(w_0)$, then the slopes of $U$ acting on overconvergent modular forms of weight $\kappa$ are the arithmetic progression $0, v, 2v, 3v, 4v, \ldots,$ each appearing with multiplicity 1.
\end{theoremmain}
\begin{proof}
By Proposition \ref{newton-prop}, the Newton polygon of the characteristic power series of $U$ has vertices $(\alpha, \frac{1}{2}\alpha(\alpha - 1)v)$ and slopes $0, v, 2v, 3v, 4v, \ldots.$
\end{proof}

As in the $p=2$ case, we see that the eigencurve is geometrically the disjoint union of countably many annuli over the boundary of weight space.

\section{Other Work} \label{sec-other}

Daniel Jacobs' thesis \cite{jacobs:thesis} uses a different approach to compute the slopes of $U_3$ on spaces of overconvergent modular forms.  He begins with a specific definite quaternion algebra, ramified at $2$ and infinity, and then uses the Jacquet-Langlands correspondence to derive results about $U_3$.  As a consequence of this different methodology, he only obtains a subset of the slopes listed in Thereom \ref{theorem-main}.  In addition, the fact that his quaternion algebra is ramified at $2$ introduces level structure at $2$ beyond just $\Gamma_0(3)$.  However, his methods are not subject to the restriction on weight that Theorem \ref{theorem-main} are: he can find slopes at weight $x \mapsto x^3$ for example.

Loeffler \cite{loeffler:pre08} computes the slopes of the $U$ operator for $p=3$, but only for weight 0.

\bibliographystyle{plain}
\bibliography{bibliography/Biblio}

\end{document}